\documentclass[10pt,a4paper]{article}

\usepackage{a4wide}
\usepackage{amsthm}
\usepackage{amsfonts}
\usepackage{amsmath}
\usepackage[english]{babel}

\usepackage{faktor}
\usepackage{etex}
\usepackage{hyperref}
\usepackage{multimedia}
\usepackage[all]{xy}
\usepackage{amssymb}
\usepackage{graphicx,textpos}
\usepackage[dvipsnames]{xcolor}
\usepackage{helvet}
\usepackage{stmaryrd}
\usepackage{enumerate}
\usepackage{todonotes}
\usepackage{pdfsync}
\usepackage{dsfont}
\usepackage[shortcuts]{extdash}
\usepackage[all]{xy}
  \newdir{ >}{{}*!/-9pt/@{>}}

\newcommand{\N}{\mathbb{N}}
\newcommand{\R}{\mathbb{R}}
\newcommand{\Q}{\mathbb{Q}}
\newcommand{\Z}{\mathbb{Z}}
\newcommand{\C}{\mathbb{C}}

\newcommand{\I}{\mathds{1}}

\newcommand{\rst}[1]{\ensuremath{{\mathbin\mid}\raise-.5ex\hbox{$#1$}}}
\newcommand{\lie}{\mathfrak{g}}
\newcommand{\lien}{\mathfrak{n}}

\DeclareMathOperator{\GL}{GL}

\DeclareMathOperator{\Aut}{Aut}
\DeclareMathOperator{\Aff}{Aff}

\DeclareMathOperator{\lat}{lat}

\author{Karel Dekimpe and Jonas Der\'e\thanks{The second author is supported by a Ph.D.~fellowship of the Research Foundation -- Flanders (FWO).
Research supported by the research Fund of the KU Leuven}
}
\title{\bf Expanding maps and non-trivial self-covers on infra-nilmanifolds}
\date{}

\newtheorem{Def}{Definition}[section]
\newtheorem{Ex}{Example}[section]
\newtheorem{Cor}[Def]{Corollary}
\newtheorem{Thm}[Def]{Theorem}
\newtheorem{Prop}[Def]{Proposition}
\newtheorem{Lem}[Def]{Lemma}
\newtheorem*{Rmk}{Remark}
\newtheorem*{Prop*}{Proposition}
\newtheorem*{Lem*}{Lemma}

\newtheorem*{con}{Conjecture}

\begin{document}
\maketitle
\begin{abstract}
Every expanding map on a closed manifold is topologically conjugate to an expanding map on an infra-nilmanifold, but not every infra-nilmanifold admits an expanding map. In this article we give a complete algebraic characterization of the infra-nilmanifolds admitting an expanding map. We show that, just as in the case of Anosov diffeomorphisms, the existence of an expanding map depends only on the rational holonomy representation of the infra-nilmanifold. A similar characterization is also given for the infra-nilmanifolds with a non-trivial self-cover, which corresponds to determining which almost-Bieberbach groups are co-Hopfian. These results provide many new examples of infra-nilmanifolds without non-trivial self-covers or expanding maps.
\end{abstract}

The study of certain chaotic dynamical systems on closed manifolds like expanding maps or Anosov diffeomorphisms is closely related to the study of infra-nilmanifolds. These manifolds are constructed as quotients of connected and simply connected nilpotent Lie groups by affine transformations. Expanding maps and Anosov diffeomorphisms on these manifolds correspond to certain group morphisms of their fundamental group and this makes it possible to study them in an algebraic way. 

An expanding map on a closed Riemannian manifold $M$ is a differentiable map $f: M \to M$ such that there exists constants $c>0$ and $\lambda > 1$ with $\Vert Df^n(v) \Vert \geq c \lambda^n \Vert v \Vert$ for all $v \in TM$ and $n > 0$. In \cite{grom81-1}, M. Gromov showed that every expanding map on a closed manifold is topologically conjugate to an expanding map on an infra-nilmanifold. Anosov diffeomorphisms are defined in a similar way (see \cite{dv11-1} for an exact definition) and it is conjectured that also every Anosov diffeomorphism on a closed manifold is topologically conjugate to an Anosov diffeomorphism on an infra-nilmanifold. There are some partial results for this conjecture, see for example \cite{mann74-1} and \cite{newh70-1}. This motivates the problem of classifying all infra-nilmanifolds admitting an expanding map or an Anosov diffeomorphism, which was already raised by Smale in \cite{smal67-1}.

The fundamental group of an infra-nilmanifold determines a rational nilpotent Lie algebra and a finite group of automorphisms of this Lie algebra, called the rational holonomy representation. It is well known that the existence of an Anosov diffeomorphism on an infra-nilmanifold depends only on this Lie algebra and rational holonomy representation, see \cite[Theorem A.]{dv09-1}. The proof of this result strongly uses the fact that every Anosov diffeomorphism induces an automorphism on the fundamental group. In this paper, we show that also the existence of an expanding map depends only on the Lie algebra and the rational holonomy representation, see Theorem \ref{expinfra}. By using the same methods we deduce a similar statement for the existence of non-trivial self-covers on infra-nilmanifolds in Theorem \ref{covinfra}. In fact both results show that the existence is equivalent to the existence of certain gradings of the Lie algebra, preserved by every automorphism of the rational holonomy representation. The techniques of this paper give a simplified proof of the result about Anosov diffeomorphisms as well.

Up till now, the only known examples of infra-nilmanifolds without a non-trival self-cover are constructed from nilpotent Lie algebras with only automorphisms of determinant $1$, see \cite{bele03-1}. From the main theorem we deduce a general way of constructing examples and this provides new examples where the corresponding Lie algebra is not of this type. Another consequence is that the existence of an expanding map or a non-trivial self-cover on a nilmanifold is invariant under commensurability of the fundamental group, answering a question of \cite{bele03-1}. Finally, our main results make it possible to determine examples of minimal dimension for low nilpotency classes.

In Section \ref{intro}, we recall some basic facts about infra-nilmanifolds and their fundamental groups, the almost-Bieberbach groups. Next, we study injective group morphisms of these almost-Bieberbach groups in Section \ref{endos} and the relation between graded Lie algebras and expanding automorphisms in Section \ref{SGrad}. These technical results are then combined to prove our main results in Section \ref{algexp} and \ref{alghop}. Finally, Section \ref{Exam} demonstrates how to apply these results.

\section{Rational holonomy representation of an infra-nilmanifold}

\label{intro}

In this section, we briefly recall what an infra-nilmanifold is and how to construct differentiable maps on these manifolds. For more details, we refer to \cite{deki96-1}. Also, we introduce the rational holonomy representation, which plays a crucial role in the formulation of our main theorems.

Let $G$ be a connected and simply connected nilpotent Lie group and $\Aut(G)$ the group of continuous automorphisms of $G$. The affine group $\Aff(G) = G \rtimes \Aut(G)$ acts on $G$ in the following natural way:
$$\forall \alpha =  (g,\delta) \in \Aff(G), \hspace{1mm} \forall h \in G: \hspace{1mm} {}^{\alpha}h = g \delta(h).$$
Let $C$ be a maximal compact subgroup of $\Aut(G)$ and consider the subgroup $G \rtimes C$ of $\Aff(G)$. A discrete, torsion-free subgroup $\Gamma \subseteq G \rtimes C$ with compact quotient $\Gamma \backslash G$ is called an almost-Bieberbach group. The quotient space $\Gamma \backslash G$ is then a closed manifold with fundamental group $\Gamma$ and we call such a manifold an infra-nilmanifold. Denote by $p: \Aff(G) \to \Aut(G)$ the projection on the second component. It is well known that the group $F = p(\Gamma)$ is a finite subgroup of $\Aut(G)$ and we refer to this group as the holonomy group of the infra-nilmanifold $\Gamma \backslash G$. 

A nilmanifold is an infra-nilmanifold with trivial holonomy group or equivalently a quotient space $\Gamma \backslash G$ with $\Gamma \subseteq   G \rtimes \{1\}\approx G$ an uniform lattice of $G$. For a general almost-Bieberbach group $\Gamma$, the subgroup $N = \Gamma \cap G$ is an uniform lattice in $G$ and so every infra-nilmanifold is finitely covered by a nilmanifold, hence its name. The subgroup $N$ is the maximal normal nilpotent subgroup of $\Gamma$ and thus equal to the Fitting subgroup of $\Gamma$. This also shows that every almost-Bieberbach group fits in an exact sequence:
\begin{align}\label{seq}
\xymatrix{1 \ar[r] & N \ar[r]  & \Gamma \ar[r] & F \ar[r] & 1}.
\end{align}
The uniform lattice $N$ is a finitely generated, torsion-free and nilpotent group and every group satisfying these three properties is called an $\mathcal{F}$-group. Conversely, every $\mathcal{F}$-group is isomorphic to an uniform lattice in some connected and simply connected nilpotent Lie group $G$. 

\smallskip

Assume that $\alpha \in \Aff(G)$ is an affine transformation satisfying $\alpha \Gamma \alpha^{-1} \subseteq \Gamma$, then $\alpha$ induces a differentiable map $$\bar{\alpha}: \Gamma \backslash G \to \Gamma \backslash G: g \mapsto {}^\alpha g$$ and this map is called an affine infra-nilendomorphism. If $\alpha \Gamma \alpha^{-1} = \Gamma$, then also $\alpha^{-1}$ induces a differentiable map $\Gamma \backslash G \to \Gamma \backslash G$ which is the inverse of $\bar{\alpha}$. This implies that $\bar{\alpha}$ is a diffeomorphism and in this case we call $\bar{\alpha}$ an affine infra-nilautomorphism. The eigenvalues of an affine infra-nilendomorphism $\bar{\alpha}$ are defined as the eigenvalues of the linear part of $\alpha$. We call an affine infra-nilendomorphism expanding if it only has eigenvalues of absolute value $>1$. With these notations we can formulate the exact statement about expanding maps of Gromov:  

\begin{Thm}[Gromov]
\label{Gromov}
Every expanding map on a closed manifold is topologically conjugate to an expanding affine infra-nilendomorphism.
\end{Thm}
\noindent So the problem of classifying all infra-nilmanifolds admitting an expanding map is translated into the question of classifying all infra-nilmanifolds admitting an expanding affine infra-nilendomorphism. To solve this problem, we need to know more about the algebraic structure of the almost-Bieberbach groups $\Gamma$. From the exact sequence (\ref{seq}), it follows that the group $\Gamma$ contains a normal subgroup of finite index which is an $\mathcal{F}$-group, but this exact sequence does not split in general. We can embed the exact sequence (\ref{seq}) in a split exact sequence, but to do this, we need to recall some properties of \cite[Chapter 6]{sega83-1} about $\mathcal{F}$-groups. 

Let $N$ be an $\mathcal{F}$-group, then we can assume that $N$ is an uniform lattice of a connected and simply connected nilpotent Lie group $G$ with corresponding Lie algebra $\lie$. The exponential map $\exp: \lie \to G$ is a diffeomorphism and we denote by $\log: G \to \lie$ its inverse. The rational span of $\log(N)$, denoted as $\lien$, is a rational form of the real Lie algebra $\lie$ and the corresponding nilpotent group $N^\Q = \exp(\lien)$ is the radicable hull or rational Mal'cev completion of the group $N$. The automorphism groups $\Aut(N^\Q)$ and $\Aut(\lien)$ are naturally isomorphic via the exponential map and in the remaining part of this paper we will identify these two groups without further mentioning the exponential map. It's a general fact that every finite dimensional rational nilpotent Lie algebra occurs in this way and that two $\mathcal{F}$-groups have isomorphic radicable hull if and only if they are abstractly commensurable. A finitely generated subgroup $N^\prime$ of $N^\Q$ such that $N^\Q$ is also the radicable hull of $N^\prime$ (i.e.\ if $\log(N^\prime)$ spans $\lien$ as a rational vector space) is called a full subgroup. 

In general, $\log(N)$ is not a Lie subring or even an additive subgroup of $\lie$ and therefore the following definition is important:
\begin{Def}
An $\mathcal{F}$-group $N$ is called a lattice group if $\log(N)$ is an additive subgroup of $\lie$. 
\end{Def}
\noindent This means that for a lattice group $N$ we can always find a basis for $\lien$ such that $\log(N)$ is equal to the $\Z$-span of this basis.
It is well known that every $\mathcal{F}$-group is a subgroup of finite index in some lattice group. Since the intersection of lattice groups is again a lattice group, there exists a smallest lattice group containing $N$ and we will denote this lattice group as $N^{\lat}$. If $\varphi: N \to N$ is an injective group morphism, then there exists an unique extension $\varphi^\Q: N^\Q \to N^\Q$ of $\varphi$. Since $\varphi$ is injective, also $\varphi^\Q$ is injective and thus the injective group morphisms can be seen as a subset of $\Aut(N^\Q) = \Aut(\lien)$. Therefore it makes sense to talk about the determinant and characteristic polynomial of an injective group morphism $\varphi: N \to N$. Since $\varphi^\Q$ maps lattice groups to lattice groups, the following proposition is immediate (see also \cite[Lemma 4.1.]{bele03-1}):
\begin{Prop}
\label{extension}
Let $N$ be an $\mathcal{F}$-group and $\varphi: N \to N$ an injective group morphism. Then the following are true:
\begin{enumerate}
\item $\varphi^\Q(N^{\lat}) = \left(\varphi(N)\right)^{\lat} \subseteq N^{\lat}$;
\item $ [N:\varphi(N)]=[N^{\lat}:\varphi^{\Q}(N^{\lat})] = \vert \det \varphi^\Q \vert$.
\end{enumerate}
\end{Prop}
The first statement shows that every injective group morphism of $N$ also induces an injective group morphism on $N^{\lat}$ and we will denote the induced map as $\varphi^{\lat}$.



\smallskip

Since every automorphism of $N$ has a unique extension to an automorphism of $N^\Q$, we can construct the following commutative diagram from the exact sequence (\ref{seq}) where the horizontal sequences are exact:
\begin{align*}
\xymatrix{ 1 \ar[r] & N \ar[r]\ar@{ >->}[d] & \Gamma \ar[r]\ar@{
>->}[d] &
F \ar[r] \ar@{=}[d] &1\\
1 \ar[r] & N^\Q \ar[r] & \Gamma^\Q \ar[r] &F \ar[r] & 1. }
\end{align*}
The lower exact sequence splits and by fixing a splitting morphism $s: F \to \Gamma^\Q$, we get an injective morphism $F \to \Aut(N^\Q) = \Aut(\lien)$. This morphism is called the rational holonomy representation and we will always identify the holonomy group $F$ with its image under this morphism, so $F \le \Aut(\lien)$. Under this identification, we get that $\Gamma \subseteq N^\Q \rtimes F$. 

\smallskip

By using the rational holonomy representation, we have the following version of the Generalized Second Bieberbach Theorem:
\begin{Thm}
\label{Bieberbach}
Let $\varphi: \Gamma \to \Gamma^\prime$ be an isomorphism between two almost-Bieberbach groups with commensurable Fitting subgroups $N$ and $N^\prime$, then there exists an affine transformation $\alpha \in N^\Q \rtimes \Aut(N^\Q)$ such that $\varphi(\gamma) = \alpha \gamma \alpha^{-1}$. 
\end{Thm}
\noindent The proof of this version of the Second Bieberbach Theorem is identical to the proof given in \cite[Section 2.2]{deki96-1} by replacing the Lie group $G$ by the radicable hull $N^\Q$. Note that the theorem also works for injective group morphisms $\Gamma \to \Gamma$, since the image is a subgroup of finite index in that case.

\section{Injective group morphisms of $\mathcal{F}$-groups}
\label{endos}

In this section we study the relation between injective group morphisms of an $\mathcal{F}$-group and automorphisms of the corresponding rational Lie algebra. Let $N$ be an $\mathcal{F}$-group with Lie algebra $\lien$ and $\varphi: N \to N$ an injective group morphism, then we know from Section \ref{intro} that $\varphi$ has a unique extension $\varphi^\Q: \lien \to \lien$. The question we study here is which automorphisms of $\lien$ induce a group morphism on $N$, i.e.\ for which automorphisms $\varphi^\Q \in \Aut(\lien)$ we have $\varphi^\Q(N) \subseteq N$. 

\smallskip

First we determine the properties of the extension $\varphi^\Q \in \Aut(\lien)$. From Proposition \ref{extension} we know that $\varphi$ induces an injective group morphism $\varphi^{\lat}: N^{\lat} \to N^{\lat}$ and it's obvious that $\varphi^\Q$ is also the unique extension of $\varphi^{\lat}$. By fixing a basis in $\log(N^{\lat})$ for $\lien$, we get an isomorphism from $\lien$ to $\Q^n$ as vector spaces such that this isomorphism maps $\log(N^{\lat})$ to $\Z^n$. The linear map $\varphi^\Q$ then maps $\Z^n$ to $\Z^n$ and this implies that the characteristic polynomial of $\varphi^\Q$ lies in $\Z[X]$. 

If we also assume that $\varphi$ is an automorphism of $N$, then $\varphi^{\lat}$ is an automorphism of $N^{\lat}$ and thus we get that $\varphi^\Q(\Z^n) = \Z^n$. So for every automorphism $\varphi \in \Aut(N)$ we get that $\vert\det(\varphi^\Q)\vert = 1$. An automorphism of $\lien$ with characteristic polynomial in $\Z[X]$ and determinant $\pm 1$ is called integer-like (see \cite{dd13-1}), so we conclude that $\varphi^\Q$ is integer-like in this case. It's easy to see that the converse is not true, for example the matrix $$\begin{pmatrix}\frac{5}{2} & \frac{1}{2} \\[4pt] \frac{1}{2}& \frac{1}{2}  \end{pmatrix}$$ is integer-like but doesn't map $\Z^2$ to $\Z^2$. 

\smallskip

So although not every integer-like automorphism of $N^\Q$ induces an automorphism of $N$, there is a partial converse by taking some power of the automorphism (see \cite[Theorem 3.4.]{deki99-1}):
\begin{Thm}
\label{PowerAuto}
Let $N$ be an $\mathcal{F}$-group with corresponding radicable hull $N^\Q$. If $\varphi \in \Aut(N^\Q)$ is integer-like then there exists some $k > 0$ such that $\varphi^k(N) = N$.
\end{Thm}
\noindent For example, the third power of the matrix above is equal to $\begin{pmatrix} 17 & 4\\4 & 1 \end{pmatrix} \in \GL(n,\Z)$. This theorem is crucial for the study of Anosov diffeomorphisms on infra-nilmanifolds. Unfortunately, a similar theorem cannot be true in the more general case of automorphisms $\varphi \in \Aut(N^\Q)$ with characteristic polynomial in $\Z[X]$, even not in the abelian case, as we can see from the following example:
\begin{Ex}
Consider the matrix 
\begin{align*} A = \begin{pmatrix}
\frac{1}{2}& \frac{1}{2} \\[4pt]
-\frac{3}{2}& \frac{5}{2}
\end{pmatrix}
\end{align*}
with characteristic polynomial $p(X) = X^2 - 3X + 2$ and the vector $v = \begin{pmatrix}0 \\ 1 \end{pmatrix}$. A computation shows that 
\begin{align*}
A^k(v) = v + \frac{1}{2} \sum_{i=1}^k 2^{i-1} \begin{pmatrix}1 \\ 3 \end{pmatrix}
\end{align*}
and thus $A^k(v) \notin \Z^2$ for all $k > 0$. This shows us that the condition $\vert \det (\varphi) \vert =1$ is really necessary in Theorem \ref{PowerAuto}. 
\end{Ex}

To simplify notations, we will also denote by $\varphi \in \Aut(N^\Q)$ the unique extension of an injective group morphism $\varphi: N \to N$. The main result of this section shows that we can partially restore Theorem \ref{PowerAuto} by putting constraints on the determinant of $\varphi$: 
\begin{Thm}
\label{primes}
Let $N_1$ and $N_2$ be $\mathcal{F}$-groups with identical radicable hull $N^\Q$. Then there exists a finite number of primes $p_1, \ldots, p_l$ such that for every injective group morphism $\varphi: N_1 \to N_1$ with $p_j \nmid \det (\varphi)$ for all $j \in \{1, \ldots, l\}$, there exists some $k > 0$ such that $\varphi^k(N_2) \subseteq N_2$.
\end{Thm}
\noindent Note that every integer-like automorphism of $N^\Q$ induces an automorphism on some full subgroup of $N^\Q$, so this theorem is really a generalization of Theorem \ref{PowerAuto}. 

\smallskip

The proof of this theorem uses a few technical lemmas. The first one gives us information about the index of the intersection of subgroups:
\begin{Lem}
\label{index}
Let $H$ be any group with finite index subgroups $K_1$ and $K_2$ of index respectively $k_1$ and $k_2$. If $\gcd(k_1,k_2) = 1$, then $K_1 \cap K_2$ is a subgroup of index $k_1k_2$ and if $K_1^\prime$ is a normal subgroup of index $k_1$ such that $K_1 \cap K_2  = K_1^\prime \cap K_2$, then $K_1 = K_1^\prime$.
\end{Lem}
\begin{proof}
For the first statement, note that $[H:K_1\cap K_2] = [H:K_1] [K_1:K_1 \cap K_2] = [H:K_2][K_2:K_1 \cap K_2]$ and thus $k_j \mid [H:K_1\cap K_2]$ for $j\in \{1,2\}$. Since $\gcd(k_1,k_2) = 1$ and $[H:K_1\cap K_2] \leq [H:K_1][H:K_2]$, we have that $[H:K_1 \cap K_2] = k_1 k_2$. Note that this is equivalent to the fact that $[K_1: K_1 \cap K_2] = k_2$.

For the second statement, we assume that $K_1 \cap K_2 =  K_1^\prime \cap K_2$ and we show that $K_1 \cap K_1^\prime$ is a subgroup of index $1$ in $K_1$. As $K_1^\prime$ is a normal subgroup, $K_1 K_1^\prime$ is a subgroup of $H$. Since the index of $K_1 \cap K_1^\prime$ in $K_1$ is equal to the index of $K_1^\prime$ in $K_1 K_1^\prime$, it divides $k_1$. From the first statement we know that $K_1 \cap K_2 = K_1 \cap K_1^\prime \cap K_2$ is a subgroup of index $k_2$ in $K_1$. Since $K_1 \cap K_1^\prime \cap K_2 \subseteq K_1 \cap K_1^\prime \subseteq K_1$, the index of $K_1 \cap K_1^\prime$ in $K_1$ divides $k_2$. As the index of $K_1 \cap K_1^\prime$ in $K_1$ divides both $k_1$ and $k_2$, we conclude from $gcd(k_1,k_2) = 1$ that it's equal to $1$.
\end{proof}
\smallskip

As a consequence, we have the following result:
\begin{Lem}
\label{perm}
Let $H$ be any group with a finite number of normal subgroups of index $i$ and $\varphi: H \to H$ be an injective group morphism such that $\varphi(H)$ is a subgroup of finite index in $H$ and $\gcd([H:\varphi(H)],i) = 1$. If we write the distinct normal subgroups of index $i$ as $K_1, \ldots, K_m$, then there exists a permutation $\pi \in S_m$ such that $\varphi(K_j) = K_{\pi(j)} \cap \varphi(H)$ for all $j \in \{1, \ldots, m\}$. 
\end{Lem}
\begin{proof}
By the previous lemma, we know that $K_j \cap \varphi(H)$ is a normal subgroup of index $i$ in $\varphi(H)$ and that if $K_j \neq K_{j^\prime}$ also $K_j \cap \varphi(H) \neq K_{j^\prime} \cap \varphi(H)$. Since $\varphi(H)$ is isomorphic to $H$, there are exactly $m$ different normal subgroups of index $i$ and thus they are all of the form $K_j \cap \varphi(H)$. Since also $\varphi(K_1), \ldots, \varphi(K_m)$ are $m$ distinct normal subgroup of index $i$ in $\varphi(H)$, we find a permutation $\pi \in S_m$ such that $\varphi(K_j) = K_{\pi(j)} \cap \varphi(H)$.
\end{proof}
\noindent Note that the first condition of Lemma \ref{perm} is satisfied for every index $i$ in all $\mathcal{F}$-groups. Also, as we mentioned in Section \ref{intro}, the image of an injective group morphism is always of finite index for these groups (and the index is given by the absolute value of the determinant, see Proposition \ref{extension}).

Take the notations of Lemma \ref{perm} and $\pi \in S_m$ the permutation corresponding to $\varphi: H \to H$. Then for any $k > 0$, we have that $$\varphi^k(K_j) = \varphi^{k-1}\left( K_{\pi(j)} \cap \varphi(H) \right) = \varphi^{k-1}\left(K_{\pi(j)}\right) \cap \varphi^k(H)$$ since $\varphi$ is injective and thus by induction we get that $\pi^k$ is the permutation corresponding to $\varphi^k$.
\smallskip

The last lemma we need is just the abelian version of Theorem \ref{primes}. Recall that a lattice of a vector space is a discrete subgroup which spans the vector space.
\begin{Lem}
\label{lattice}
Let $L$ be a lattice in $\Q^n$, then there exists primes $p_1, \ldots, p_l$ such that for every $A \in \GL(n,\Q)$ with $A(\Z^n) \subseteq \Z^n$ and $p_i \nmid \det(A)$, there exists some $k > 0$ such that $A^k(L) \subseteq L$.
\end{Lem}

\begin{proof}
Fix a basis for $\Q^n$ in $L$ and denote by $P$ the matrix of change of basis from this basis to the standard basis. Take $m$ the product of all the denominators of $P$ and $P^{-1}$ and take $p_i$ all the primes dividing $m$. Now assume $A \in \GL(n,\Q)$ with $A(\Z^n) \subseteq \Z^n$ and $p_i \nmid \det(A)$. The matrix representation of $A^k$ for the chosen basis in $L$ is given by $P^{-1} A^k P$ and thus we have to check that $P^{-1} A^k P$ has integer entries for some $k> 0$. 

From the choice of the primes $p_i$, it follows that $A$ projects to an element of $GL(n,\Z_m)$ and write this projection as $\bar{A}$. Since $\GL(n,\Z_m)$ is a finite group, there exists some $k >0$ such that $\bar{A}^k = \bar{I}_n$. This means that $m$ divides every entry of $A^k - I_n$. So we have that 
$$ P^{-1} A^k P - I_n = P^{-1} (A^k - I_n) P$$
has integer entries because of our choice of $m$ and thus also $P^{-1} A^k P$ has integer entries.
\end{proof}
\smallskip

With the help of these lemmas, we give a proof of the general version of Theorem \ref{primes}:

\begin{proof}[Proof of Theorem \ref{primes}]
Every injective group morphism $\varphi: N_1 \to N_1$ also induces an injective group morphism $\varphi^{\lat}:N_1^{\lat} \to N_1^{\lat}$ with the same determinant, so we can always assume that $N_1$ is a lattice group. By fixing a basis for $\lien$ in $\log(N_1)$, the vector space $\lien$ is isomorphic to $\Q^n$ such that $\log(N_1)$ is equal to $\Z^n$ under this isomorphism. The injective group morphisms of $N_1$ correspond to Lie algebra automorphisms which map $\Z^n$ to $\Z^n$. If we assume that also $N_2$ is a lattice group, then $\log(N_2)$ is a lattice in $\lien \approx \Q^n$ and the theorem follows from Lemma \ref{lattice}.

\smallskip

Next we show that if the theorem is true for $N_2$, it is also true for every normal subgroup $K$ of $N_2$. Let $p_j$ be the finite set of primes corresponding to $N_2$ and add all the primes that divide the index of $K$ in $N_2$. If $\varphi: N_1 \to N_1$ is a group morphism satisfying the conditions of the theorem, then we can assume $\varphi(N_2) \subseteq N_2$ by taking some power of $\varphi$. Since $\varphi$ is an injective group morphism of $N_2$, we can apply Lemma \ref{perm}. So take $K_1 = K, \ldots, K_m$ all normal subgroups of the same index as $K$ in $N_2$ and by Lemma \ref{perm}, we know that there exists $\pi \in S_m$ with $\varphi(K_j) = K_{\pi(j)} \cap \varphi(N_2)$. Now take $k > 0$ such that $\pi^k = 1$, then $$\varphi^k(K) = \varphi^k(K_1) = K_{\pi^k(1)} \cap \varphi^k(N_2) = K_1 \cap \varphi^k(N_2) \subseteq K_1 = K$$ and therefore $\varphi^k(K) \subseteq K$ as we wanted.

\smallskip

For the proof in the general case, we know that $N_2 \subseteq N_2^{\lat}$ is a subgroup of finite index. Since $N_2^{\lat}$ is a nilpotent group, $N_2$ is a subnormal subgroup, meaning we can find subgroups $H_0 = N_2 \le H_1 \le \ldots \le H_m = N_2^{\lat}$ with $H_j$ normal in $H_{j+1}$. By iterating the result of the previous paragraph, we conclude that the theorem also holds for $N_2$.
\end{proof}

\noindent From the proof it also follows that we can give the set of primes explicitly by looking at the index of $N_2$ in $N_2^{\lat}$ and the matrix of change of basis from $N_1^{\lat}$ to $N_2^{\lat}$.

\begin{Rmk}
For the particular case of integer-like automorphisms of $N^\Q$, the technicalities of Lemma \ref{perm} are avoided. So the proof of Theorem \ref{primes} gives us a simplified proof of Theorem \ref{PowerAuto}, avoiding the use of Jonqui\`ere groups as in \cite{deki99-1}.
\end{Rmk}

\section{Graded Lie algebras}
\label{SGrad}

In this section we introduce graded Lie algebras and show why they are important for studying expanding maps. The most important result states that a rational Lie algebra is positively graded if and only if it has an expanding automorphism, which was already known for real and complex Lie algebras. In this way, starting from an expanding automorphism and an arbitrary prime, we can construct an expanding automorphism with determinant a power of this prime, which is important for using Theorem \ref{primes}. All Lie algebras we consider in this section are finite dimensional.

\begin{Def}
A rational Lie algebra $\lien$ is called graded if it is endowed with a decomposition as a direct sum $$\displaystyle \lien = \bigoplus_{i \in \Z} \lien_i$$ of subspaces $\lien_i \subseteq \lien$ such that $[\lien_i,\lien_j] \subseteq \lien_{i+j}$ for all $i,j \in \Z$. The decomposition into subspaces $\lien_i$ is called a grading for the Lie algebra $\lien$. The grading is called positive if $\lien_i = 0$ for all $i \leq 0$. We say that an automorphism $\varphi \in \Aut(\lien)$ preserves the grading if $\varphi(\lien_i) = \lien_i$ for all $i \in \Z$.
\end{Def}
\noindent The existence of a grading $\lien = \displaystyle \bigoplus_{i \in \Z} \lien_i$ with $\lien_0 = 0$ implies that the Lie algebra $\lien$ is nilpotent. In the remaining part of this article we are only interested in nilpotent Lie algebras, since these are closely related to (infra-)nilmanifolds.

Sometimes a grading is defined as a decomposition \begin{align*} \lien = \bigoplus_{r \in \R} \lien_r\end{align*} which is indexed by real numbers rather than integers as in our definition. These definitions are equivalent because of \cite[Proposition 2.6.]{dip03-1}, even in the case of positive gradings and we call a grading indexed by real numbers an $\R$-grading. In fact, the proof of \cite[Proposition 2.6.]{dip03-1} shows that starting from an $\R$-grading, the grading of $\lien$ is formed by just renaming the subspaces $\lien_r$ with integer indices. We will use this equivalence of definitions later on, as it is often easier to construct an $\R$-grading for a Lie algebra. 

\smallskip

Graded Lie algebras are an important tool to study expanding maps. For real or complex Lie algebra's, \cite[Theorem 2.7.]{dl03-1} shows that the existence of an expanding automorphism is equivalent to the existence of a positive $\R$-grading (and thus a positive grading). For rational Lie algebras, it's easy to see that the existence of a grading induces an expanding automorphism: if $\lien_i$ is the decomposition, then for every $q \in \Q$ with $q > 1$ there exists an expanding automorphism map which maps an element $x \in \lien_i$ to $q^i x$. To translate this fact to $\mathcal{F}$-groups, we first need the following proposition:
\begin{Prop}
\label{constructgroup}
Let $\lien$ be a rational nilpotent Lie algebra and $N^\Q$ the corresponding nilpotent group. Identify the group $\Aut(\lien) = \Aut(N^\Q)$ as a subgroup of $\GL(n,\Q)$ by fixing a basis for $\lien$ as a vector space. Then there exists a full subgroup $N$ of $N^\Q$ such that every automorphism with integer entries induces an injective group morphism on $N$, i.e.\ if $\varphi \in \Aut(N^\Q) \subseteq \GL(n,\Q)$ has integer entries, then $\varphi(N) \subseteq N$.
\end{Prop}

\begin{proof}
Denote the basis for $\lien$ as $\{v_1, \ldots, v_n \}$ and take $N$ the full subgroup of $N^\Q$ generated by the elements $\exp(v_1), \hspace{0.5mm} \ldots  \hspace{0.5mm} , \exp(v_n)$. Since $N$ is an $\mathcal{F}$-group, we can consider $N^{\lat}$ and we claim that $N^{\lat}$ satisfies the properties of the Proposition. 

Let $\varphi$ be an automorphism of $\Aut(\lien)$ with integer entries in the basis $\{v_1, \ldots, v_n\}$, then we have to show that $\varphi(N^{\lat}) \subseteq N^{\lat}$. Since $\varphi$ has integer entries, we have that $\varphi(N) \subseteq N^{\lat}$, because $\varphi(v_i)$ is in the $\Z$-span of the basis $\{v_1, \ldots, v_n\}$ and $N^{\lat}$ is a lattice group. By Proposition \ref{extension} it follows that $$\varphi(N^{\lat}) = \left(\varphi\left(N\right)\right)^{\lat} \subseteq \left(N^{\lat}\right)^{\lat}$$ and the latter is of course equal to $N^{\lat}$. We conclude that $\varphi(N^{\lat}) \subseteq N^{\lat}$.
\end{proof}

It now easily follows that we can construct many expanding group morphisms if the corresponding Lie algebra is graded:

\begin{Cor}
\label{fp}
Let $\lien$ be a rational nilpotent Lie algebra with a positive grading and $N^\Q$ the corresponding nilpotent group. Then there exists a full subgroup $N$ and some $k >0$ such that for every prime $p$, there exists an injective group morphism $\varphi_p: N \to N$ with $\det (\varphi_p) = p^k$ which is also expanding. Moreover, these $\varphi_p$ commute with every automorphism that preserves the grading.
\end{Cor}

\begin{proof}
Fix a positive grading $\lien_i$ for the rational Lie algebra $\lien$ corresponding to $N^\Q$ and fix a basis for $\lien$ such that every basis vector lies in a subspace $\lien_i$. For every prime $p$ there exists an automorphism $\varphi_p: \lien \to \lien$ such that $\varphi_p(x) = p^i x $ for all $x \in \lien_i$. From our choice of basis vectors, it follows that $\varphi_p$ has integer entries in the basis and that $\varphi_p$ commutes with every automorphism that preserves the grading. Note that every $\varphi_p$ is expanding since the grading is positive and $\det(\varphi_p) = p^k$ for some fixed $k >0$. Now take an $\mathcal{F}$-group as in Proposition \ref{constructgroup}, then the statement follows immediately.
\end{proof}

\smallskip

So starting from a nilpotent Lie algebra with positive grading, we can construct many expanding group morphisms on a full subgroup of the corresponding radicable hull. The main result of this section shows that for rational Lie algebras every expanding automorphism also induces a positive grading. The exact formulation is as follows:

\begin{Thm}
\label{Grading}
Let $\lien$ be a rational Lie algebra with an expanding automorphism $\varphi \in \Aut(\lien)$, then $\lien$ has a positive grading which is preserved by $\varphi$. Moreover if $\varphi$ commutes with an automorphism $\psi \in \Aut(\lien)$, then $\psi$ preserves the positive grading. 
\end{Thm}

The proof of this theorem is based on some facts about algebraic number fields, so let us recall the most important definitions and properties. A number field is a field extension $\Q \subseteq E$ of finite degree. Let $E$ be a number field and $n$ the degree of the field extension, then there exist exactly $n$ monomorphisms $E \to \C$ and we write these monomorphisms as $\sigma_i: E \to \C$ with $i=1,\ldots, n$. The $E$-conjugates of $\alpha \in E$ are the complex numbers $\sigma_i(\alpha)$. For every $\alpha \in E$, we define the norm $N_E(\alpha)$ as the product of all its $E$-conjugates, so
\begin{align*}
N_E(\alpha) = \prod_{i=1}^n \sigma_i(\alpha).
\end{align*}
This gives us a map $N_E: E \to \Q$ which preserves the product, i.e.\ $N_E(\alpha \beta) = N_E(\alpha) N_E(\beta)$. 
If $\alpha,\beta \in E$ are $E$-conjugate, then $N_E(\alpha) = N_E(\beta)$. An element of $E$ is called an algebraic integer if its minimal polynomial has coefficients in $\Z$ and in this case we also have $N_E(\alpha) \in \Z$.

\smallskip

Now take $A \in \GL(n,\Q)$ a diagonalizable matrix with characteristic polynomial $p(X) \in \Q[X]$. Denote by $E \subseteq \C$ the minimal splitting field of $p(X)$, which is of course an algebraic number field. We define the matrix $N(A) \in \GL(n,\C)$ by replacing each eigenvalue $\lambda$ of $A$ by $N_E(\lambda)$; so if $v_\lambda$ is an eigenvector of $A$ for the eigenvalue $\lambda \in E$, then $N(A)(v_\lambda) = N_E(\lambda) v_\lambda$. It's easy to see that if $P \in \GL(n,\Q)$, then $N(P A P^{-1}) = P N(A) P^{-1}$, so $N(A)$ is invariant under change of basis. Note that if $B$ commutes with $A$, then $B$ also commutes with $N(A)$. 

A priori we only know that $N(A) \in \GL(n,\C)$, but we show that $N(A) \in \GL(n,\Q)$ for every diagonalizable matrix $A \in \GL(n,\Q)$. First decompose the characteristic polynomial $p(X)$ of $A$ into its $\Q$-irreducible components, so $$p(X) = p_1(X) \ldots p_k(X)$$ with $p_i(X) \in \Q[X]$ irreducible over $\Q$. Denote the degree of the polynomials $p_i$ as $n_i$ Since $A$ is diagonalizable, its rational canonical form is equal to $$L = \begin{pmatrix}L(p_1) & 0 & \dots & 0\\
0 & L(p_2) & \dots & 0\\
\vdots & \vdots & \ddots & 0 \\
0 & 0 & \dots & L(p_n) \end{pmatrix},$$ so there exists $P \in \GL(n,\Q)$ such that $P A P^{-1} = L$. Now take $\lambda_i$ a root of the polynomial $p_i$. Note that $N_E(\lambda_i)$ doesn't depend on the choice of the root $\lambda_i$, since all roots of $p_i$ are $E$-conjugate. It's easy to see then that $N(L)$ is equal to $$N(L) = \begin{pmatrix}N_E(\lambda_1) \I_{n_1} & 0 & \dots & 0\\
0 & N_E(\lambda_2) \I_{n_2} & \dots & 0\\
\vdots & \vdots & \ddots & 0 \\
0 & 0 & \dots & N_E(\lambda_k) \I_{n_k} \end{pmatrix}  \in \GL(n,\Q).$$ From this it follows that $N(A) = P^{-1} N(L) P \in \GL(n,\Q)$.

Now assume $\lien$ is a rational Lie algebra and $\varphi \in \Aut(\lien)$ an automorphism. Then $N(\varphi)$ is a linear map $\lien \to \lien$ and we claim that it is also a Lie algebra automorphism. Consider the complexification $\lien^\C = \C \otimes \lien$, then it suffices to show that $N(\varphi)^\C = 1_\C \otimes N(\varphi)$ is a Lie algebra automorphism. By the linearity of the bracket, it is sufficient to show that $N(\varphi)^\C$ preserves the bracket for a basis of $\lien^\C$. Since there exists a basis of eigenvectors for $\varphi^\C$, we show that $N(\varphi)^\C$ preserves the bracket of all eigenvectors. Let $v_\lambda$ and $v_\mu$ be two eigenvectors of $\varphi^\C$, then $[v_\lambda, v_\mu]$ is an eigenvector of eigenvalue $\lambda \mu$. So $$[N(\varphi)(v_\lambda),N(\varphi)(v_\mu) ] = N_E(\lambda) N_E(\mu) [v_\lambda,v_\mu] = N_E(\lambda \mu) [v_\lambda,v_\mu] = N(\varphi) ([v_\lambda,v_\mu]),$$ with $E$ a splitting field for the characteristic polynomial of $\varphi$.

\smallskip

We conclude this discussion with the following proposition:

\begin{Prop}
Let $\varphi \in \Aut( \lien)$ be an automorphism of a rational Lie algebra. Then $N(\varphi) \in \Aut(\lien)$ is an automorphism with only rational eigenvalues and which commutes with every automorphism commuting with $\varphi$. Moreover, if $\varphi$ is expanding, then $N(\varphi)$ is also expanding. 
\end{Prop}

\begin{proof}
The only thing left to show is the last statement, namely that $N(\varphi)$ is expanding if $\varphi$ is expanding. The eigenvalues of $N(\varphi)$ are equal to $N_E(\lambda)$ with $\lambda$ an eigenvalue of $\varphi$ and $E$ the splitting field of the characteristic polynomial of $\varphi$. So it suffices to show that if $\lambda$ and all its $E$-conjugates are $>1$ in absolute value, then also $N_E(\lambda) > 1$. Since $N_E(\lambda)$ is just the product of the $E$-conjugates of $\lambda$, this follows immediately.\end{proof}

From this proposition, the proof of Theorem \ref{Grading} now follows directly:

\begin{proof}[Proof of Theorem \ref{Grading}]
Since $\Aut(\lien)$ is an algebraic group, we can always assume that $\varphi$ is diagonalizable. Then $N(\varphi)$ is an expanding automorphism with only eigenvalues in $\Q$. By squaring $\varphi$ if necessary we can assume that all eigenvalues of $N(\varphi)$ are positive. Since all eigenvalues are rational, the corresponding eigenspaces are rational subspace of $\lien$. Take the $\R$-grading $\lien = \bigoplus_{r \in \R} V_r$ where $V_r$ is the eigenspace of eigenvalue $e^r$. It's easy to see that this is indeed an $\R$-grading. Since $N(\varphi)$ is expanding, it follows that the grading is positive. The existence of a positive $\R$-grading implies the existence of a positive grading by just renaming the subspaces $V_r$, as explained above. 

If $\psi$ commutes with $\varphi$, it also commutes with $N(\varphi)$ and thus it preserves the eigenspaces of $N(\varphi)$. Since the grading is given by eigenspaces of $N(\varphi)$, the last statement follows.
\end{proof}


\section{Algebraic characterization for expanding maps}
\label{algexp}
In this section we combine the previous results to prove our algebraic characterization of infra-nilmanifolds admitting an expanding map. We start with the easier case of nilmanifolds, which follows almost directly now:

\begin{Thm}
\label{expnil}
Let $N \backslash G$ be a nilmanifold with corresponding rational Lie algebra $\lien$, then the following are equivalent:
\begin{enumerate}
\item $N \backslash G$ admits an expanding map;
\item $\lien$ has a positive grading;
\item there exists an expanding automorphism $\varphi \in \Aut(\lien)$.
\end{enumerate}
\end{Thm}

\begin{proof}
The fact that $1.$ implies $3.$ follows directly from Theorem \ref{Gromov}. In Theorem \ref{Grading} we exactly showed that $3.$ implies $2.$. So the only implication left to show is that $2.$ implies $1.$. 

To show this implication, fix a positive grading of $\lien$ and a full subgroup $N_0$ of $N^\Q$ as in Corollary \ref{fp}. Since $N_0$ and $N$ have the same radicable hull, there exists primes $p_1,\ldots,p_l$ as in Theorem \ref{primes}. Now take $p$ a prime such that $p \neq p_j$ for all $j \in \{1, \ldots, l\}$ and the group morphism $\varphi_p: N_0 \to N_0$ as in Corollary \ref{fp}. Because of our choice of $p$, there exists some power of $\varphi_p$ such that $\varphi_p^k(N) \subseteq N$ and the induced map on $N \backslash G$ is an expanding map.
\end{proof}

\noindent In particular, this theorem implies that the existence of an expanding map on a nilmanifold is invariant under commensurability of the fundamental group. This answers a question of \cite{bele03-1}. This theorem is also the first step towards another problem stated in \cite{bele03-1}, namely the question if the existence of an expanding map depends only on the complexification $\C \otimes \lien$ of the rational Lie algebra $\lien$. This problem is now translated to the following: if a complex Lie algebra has a positive grading, does every rational form of this Lie algebra have a positive grading?

\smallskip

For infra-nilmanifolds, we have a similar theorem but with an extra condition coming from the rational holonomy representation:

\begin{Thm}
\label{expinfra}
Let $\Gamma \backslash G$ be an infra-nilmanifold with associated rational holonomy representation $F \le \Aut(\lien)$. Then the following are equivalent:
\begin{enumerate}
\item $\Gamma \backslash G$ admits an expanding map;
\item $\lien$ has a positive grading, preserved by every automorphism in $F$;
\item there exists an expanding automorphism $\varphi \in \Aut(\lien)$ which commutes with every element of the holonomy group $F$.
\end{enumerate} 
\end{Thm}

\begin{proof}
The implication from $3.$ to $2.$ is again immediate from Theorem \ref{Grading}. We will first show that $1.$ implies $3.$ and conclude with the implication from $2.$ to $1.$.

\smallskip

So first assume that $\Gamma \backslash G$ admits an expanding map. Because of Theorem \ref{Gromov} and Theorem \ref{Bieberbach}, we can assume that this expanding map is given by an affine infra-nilendomorphism $\bar{\alpha}$ with $\alpha = (g,\delta) \in N^\Q \rtimes \Aut(\lien)$. It's easy to see that $\delta F \delta^{-1} = F$ and thus we can assume $\delta$ commutes with every element of $F$ by replacing $\alpha$ by some power of $\alpha$ if necessary. The automorphism $\delta$ is then the expanding automorphism we need.

\smallskip

Next assume that there exists a positive grading, preserved by every element of $F$. By Corollary \ref{fp}, we know that there exists a full subgroup $N_0$ of $N^\Q$ with for every prime $p$ an expanding group morphism $\varphi_p: N_0 \to N_0$ such that $\det( \varphi_p)$ is a power of $p$ and $\varphi_p$ commutes with every element of $F$. Write $F = \{f_1, \ldots, f_l\}$ and fix a finite number of elements $n_j \in N^\Q$ such that $$(n_1,f_1), \ldots, (n_l,f_l) \in \Gamma,$$ where we consider $\Gamma \subseteq N^\Q \rtimes F$. So every element $\gamma \in \Gamma$ can be written as $\gamma = (n,1) (n_j,f_j)$ for some $j\in \{1, \ldots, l\}$ and vice versa, if an element of $N^\Q \rtimes F$ can be written in this form, it is an element of $\Gamma$. Take $N_1$ the full subgroup of $N^\Q$ generated by $N$ and the elements $n_1, \ldots, n_l$ and take $N_2$ a normal subgroup of finite index in $N_1$ such that $N_2 \subseteq N$. Now take $\varphi = \varphi_p: N_0 \to N_0$ with $p$ different from all the primes we get by Theorem \ref{primes} for $N_1, N_2$ and $N$ (where we take the first group equal to $N_0$ each time) and also $p \nmid [N_1:N_2]$. By taking some power of $\varphi$ we can thus assume that $\varphi(N_1) \subseteq N_1, \varphi(N_2) \subseteq N_2$ and $\varphi(N) \subseteq N$. 

Consider now the group morphism that $\varphi$ induces on $\faktor{N_1}{N_2}$. We claim that this group morphism is injective (and thus an isomorphism since the group is finite). For this we have to show that $\varphi(N_1) \cap N_2 = \varphi(N_2)$. Since $N_2 \subseteq N_1$, we obviously have $\varphi(N_2) \subseteq \varphi(N_1) \cap N_2$, so it suffices to show that both subgroups have the same index in $N_2$. From Proposition \ref{extension} we know that $\varphi(N_2)$ is a subgroup of index $\vert \det \left(\varphi\right) \vert$ in $N_2$. Similarly, $\varphi(N_1)$ is a subgroup of index $\vert \det \left(\varphi\right) \vert$ in $N_1$ and by Lemma \ref{index} and our choice of $p$, we get that $\varphi(N_1) \cap N_2$ is also a subgroup of index $\vert \det( \varphi) \vert$ in $N_2$. The claim thus follows.

By taking some power of $\varphi$, we can assume that $\varphi$ induces the identity on $\faktor{N_1}{N_2}$. Equivalently, we have for every $n_j$ that $\varphi(n_j) = \tilde{n}_j n_j$ for some $\tilde{n}_j \in N_2 \subseteq N$. This implies that for every $\gamma = (n, 1) (n_j,f_j) \in \Gamma$
\begin{align*}
(1,\varphi)\gamma (1,\varphi^{-1}) &=(1,\varphi) (n, 1) (n_j,f_j)(1,\varphi^{-1}) \\ &=  (\varphi(n),1)(\varphi(n_j),\varphi f_j \varphi^{-1}) \\ & = (\varphi(n) \tilde{n}_j, 1) (n_j,f_j) \in \Gamma\end{align*}
since $\varphi(n) \tilde{n}_j \in N$. We conclude that $\varphi \Gamma \varphi^{-1} \subseteq \Gamma$ and thus $\varphi$ induces an expanding map on the infra-nilmanifold $\Gamma \backslash G$.
%
%
\end{proof}

As we explained above, every infra-nilmanifold $\Gamma \backslash G$ is finitely covered by a nilmanifold, for example the nilmanifold $N \backslash G$ corresponding to the Fitting subgroup $N$ of $\Gamma$. The lift of an expanding map is again an expanding map and thus if $\Gamma \backslash G$ admits an expanding map, also $N \backslash G$ admits one. This fact also follows trivially from Theorem \ref{expinfra}, since the holonomy group $F$ only puts extra conditions in comparison to Theorem \ref{expnil}.

The only known examples of infra-nilmanifolds not admitting an expanding map are constructed as the finite quotient of a nilmanifold which doesn't admit an expanding map. We conjecture that this is indeed the only possibility for an infra-nilmanifold without expanding map:
\begin{con}
An infra-nilmanifold $\Gamma \backslash G$ admits an expanding map if and only the nilmanifold $N \backslash G$ admits an expanding map.
\end{con}
For Lie algebras of homogeneous type (see \cite{dl03-2} for the definition) the conjecture is true, see \cite[Theorem 5.1.]{dl03-1}. Because of Theorem \ref{expnil} and Theorem \ref{expinfra}, we can reformulate this conjecture in the following form:

\begin{con}
Let $\lien$ be a rational Lie algebra and $F \le \Aut(\lien)$ a finite group of automorphisms. If there exists an expanding automorphism of $\lien$, then there also exists an expanding automorphism of $\lien$ which commutes with every element of $F$. Equivalently, if there exists a positive grading of $\lien$, then there also exists a positive grading which is preserved by every element of $F$.
\end{con}

\section{Co-Hopfian groups}
\label{alghop}

An expanding map of an infra-nilmanifold $\Gamma \backslash G$ is an example of a non-trivial self-cover, i.e.\ a covering map $\Gamma \backslash G \to \Gamma \backslash G$ which is not a homeomorphism. So Theorem \ref{expnil} and Theorem \ref{expinfra} give an algebraic way of constructing non-trivial self-covers on infra-nilmanifolds. The natural question we answer in this section is if there is an algebraic way of describing all infra-nilmanifolds which allow a non-trivial self-cover. The following definition is then natural in this discussion:

\begin{Def}
We call a group $H$ co-Hopfian if $H$ contains no non-trivial subgroup isomorphic to itself. Equivalently, $H$ is co-Hopfian if every injective group morphism $\varphi: H \to H$ is automatically surjective.
\end{Def}
\noindent So infra-nilmanifolds with only trivial self-covers correspond to almost-Bieberbach groups which are co-Hopfian. Because of Proposition \ref{extension}, an $\mathcal{F}$-group is co-Hopfian if and only if every injective group morphism has determinant $\pm 1$. 

\smallskip

Assume now that $N$ is an $\mathcal{F}$-group that is not co-Hopfian. This means there exists an injective group morphism $\varphi: N \to N$ which is not surjective, so with $\vert \det(\varphi) \vert > 1$. From Section \ref{endos} we know that $\varphi$ has characteristic polynomial $p(X) \in \Z[X]$. If we decompose $p(X)$ into its $\Q$-irreducible components, we get $p(X) = p_1(X) \ldots p_l(X)$ with $p_i(X) \in \Z[X]$ and at least one $p_i(X)$ with $\vert p_i(0) \vert > 1$. This means that $N(\varphi)$ is an automorphism with only eigenvalues in $\Z$ and at least one eigenvalue $\lambda$ with $\vert \lambda \vert > 1$. So the translation of Theorem \ref{Grading} from expanding maps to non-trivial self-covers is the following:

\begin{Thm}
Let $\lien$ be a rational Lie algebra with an automorphism $\varphi$ with characteristic polynomial in $\Z[X]$ and $\vert \det (\varphi) \vert > 1$, then $\lien$ has a non-negative and non-trivial grading. Moreover, if $\psi$ commutes with $\varphi$, then $\psi$ preserves the grading. 
\end{Thm}

\noindent Recall that a grading $\lien = \displaystyle \bigoplus_{i \in \Z} \lien_i$ is called non-negative if $\lien_i = 0$ for all $i < 0$ and non-trivial if $\lien_0 \neq \lien$. Just as in Corollary \ref{fp}, we have the following consequence of Proposition \ref{constructgroup}:

\begin{Cor}
Let $\lien$ be a rational Lie algebra with non-negative and non-trivial grading. Then there exists a full subgroup $N$ and a $k >0$ such that for every prime $p$, there exists an injective group morphism $N \to N$ with determinant $p^k$. Moreover, this group morphism commutes with every automorphism that preserves the grading.
\end{Cor}

We conclude the following theorem, which is the analog of Theorem \ref{expinfra} for non-trivial self-covers. The proof is identical as before:

\begin{Thm}
\label{covinfra}
Let $\Gamma$ be an almost-Bieberbach group with rational holonomy representation $F \le \Aut(\lien)$. Then the following are equivalent:
\begin{enumerate}
\item $\Gamma$ is not co-Hopfian;
\item $\lien$ has a non-trivial and non-negative grading which is preserved by every element of $F$;
\item there exists an automorphism $\varphi \in \Aut(n^\Q)$, commuting with every element of $F$, with characteristic polynomial in $\Z[X]$ and $\vert \det (\varphi) \vert > 1$;
\end{enumerate} 
\end{Thm}

One of the consequences of this algebraic characterization is that an almost-Bieberbach group is co-Hopfian if its Fitting subgroup is co-Hopfian. We conclude this section by giving an elementary proof of this fact. 

If $H$ is any group, we call a subgroup $K \subseteq H$ injectively characteristic if $\varphi(K) \subseteq K$ for all injective group morphisms $\varphi: H \to H$. Every injectively characteristic subgroup is a normal subgroup (since every inner automorphism is injective). The notion of an injectively characteristic subgroup is weaker than a fully characteristic subgroup and stronger than a characteristic subgroup. It's an exercise to check that the Fitting subgroup of an almost-Bieberbach group is an injectively characteristic subgroup which is in general not fully characteristic.

The following proposition shows why injectively characteristic subgroups are important, especially when studying co-Hopfian groups:
\begin{Prop}
Let $H$ be a group and $K \subseteq H$ an injectively characteristic subgroup. If both $K$ and $\faktor{H}{K}$ are co-Hopfian then also $H$ is co-Hopfian.
\end{Prop}
\begin{proof}
Let $\varphi: H \to H$ be injective group morphism of $H$. Since $K$ is injectively characteristic, we know that $\varphi(K) \subseteq K$ and thus $\varphi$ induces a injective morphism on $K$, which we call $\varphi_K$. Because $K$ is co-Hopfian, we know that $\varphi_K$ is an automorphism. Note that $\varphi$ also induces a morphism on the group $\faktor{H}{K}$ and call this induced map $\bar{\varphi}$. Thus we have the following commutative diagram:
$$
\xymatrix{
1 \ar[r] & K \ar[r] \ar[d]^{\varphi_K} & H\ar[r] \ar[d]^{\varphi} & \faktor{H}{K} \ar[r] \ar[d]^{\bar{\varphi}} & 1 \\
1 \ar[r] & K \ar[r]        	           & H\ar[r]                  & \faktor{H}{K} \ar[r]                  & 1 .
}$$
As a consequence of the 5-lemma, it suffices to show that $\bar{\varphi}$ is an automorphism. Because $\faktor{H}{K}$ is co-Hopfian, it is sufficient to show that $\bar{\varphi}$ is injective. 
\smallskip
Assume that $\bar{\varphi}(hK) = K $ for some $h \in H$. This is equivalent to saying that $\varphi(h) \in K$. Since $\varphi_K$ is surjective, there exists some $k \in K$ with $\varphi(h) = \varphi_K(k) = \varphi(k)$. From the injectivity of $\varphi$ we have that $k= h$ and thus $hK = K$. This shows that $\bar{\varphi}$ is injective. We conclude that $\varphi$ is an automorphism and thus $H$ is co-Hopfian.
\end{proof}

By applying this proposition to the Fitting subgroup of an almost-Bieberbach group (and because finite groups are always co-Hopfian), we get the following:
\begin{Cor}
Let $\Gamma$ be an almost-Bieberbach group with Fitting subgroup $N$. If $N$ is co-Hopfian, then also $\Gamma$ is co-Hopfian.
\end{Cor}
\noindent Just as for expanding maps, the only examples of infra-nilmanifolds without non-trivial self-covers are constructed from nilmanifolds without self-covers. Again we conjecture that this is the only possibility:
\begin{con}
An infra-nilmanifold $\Gamma \backslash G$ admits a non-trivial self-cover if and only the nilmanifold $N \backslash G$ admits an expanding map, with $N$ the Fitting subgroup of $\Gamma$. Equivalently, an almost-Bieberbach group $\Gamma$ is co-Hopfian if and only if its Fitting subgroup $N$ is co-Hopfian.
\end{con}
\noindent This conjecture can be translated to the following question about automorphisms of a rational Lie algebra: if there exists an automorphism $\varphi \in \Aut(\lien)$ with characteristic polynomial in $\Z[X]$ and $\vert \det(\varphi) \vert > 1$, does there also exist an automorphism with the same properties and which commutes with every element of a finite group of automorphisms? Equivalently, starting from a Lie algebra with a non-trivial and non-negative grading, does there exist a non-trivial and non-negative grading which is preserved by every element of a finite group of automorphisms. Because of the relation between this conjecture and the one for expanding maps, we believe that a proof for one of them can easily be adapted for the other case.

\section{Applications}
\label{Exam}

In this section we highlight the most important applications of our main results. First we give some explicit examples of Lie algebras corresponding to co-Hopfian $\mathcal{F}$-groups and show that they are of minimal dimension. Next we give a general way of constructing new examples starting from a $\Q$-group by using a result of \cite{bg86-1}.

\smallskip

Recall that a nilpotent Lie algebra $\lien$ is called characteristically nilpotent if every derivation of $\lien$ is a nilpotent endomorphism, meaning that all of its eigenvalues are equal to $0$. The automorphism group $\Aut(\lien)$ is an algebraic group with corresponding Lie algebra given by the derivations of $\lien$. So if $\lien$ is characteristically nilpotent, the connected component of the identity in $\Aut(\lien)$ only contains unipotent automorphisms and therefore every automorphism has determinant $1$. We conclude that if the corresponding Lie algebra of an infra-nilmanifold is characteristically nilpotent, it can never have a non-trivial self-cover. 

In \cite[Example 2.5.]{bele03-1}, the author gives an example of a $7$-dimensional nilmanifold without self-covers (starting from a characteristically nilpotent Lie algebra) and asks if this is an example of minimal dimension. The following theorem gives a positive answer: 
\begin{Thm}
\label{dimension}
All nilmanifolds of dimension $\leq 6$ admit an expanding map.
\end{Thm}

By Theorem \ref{expnil} it suffices to show that all rational Lie algebras of dimension $\leq 6$ have a positive grading and these Lie algebras have been classified, e.g.\ in \cite{degr07-1}. Let $\lien$ be a nilpotent Lie algebra of nilpotency class $\leq 2$ and consider the subspace $\lien_2 = \gamma_2(\lien) = [\lien,\lien]$. By taking any subspace $\lien_1$ such that $\lien = \lien_1 \oplus \lien_2$ as a vector space, we find a positive grading for $\lien$. So all Lie algebras of nilpotency class $2$ have a positive grading. By the work in \cite{dip03-1} we know that all $2$-generated $4$-step nilpotent and all $3$-generated $3$-step nilpotent Lie algebras have a positive grading as well. So only a few Lie algebras of \cite{degr07-1} are left to check and it's an easy exercise to construct positive gradings on these by hand. 

\smallskip

Thus the minimal dimension of a nilmanifold without non-trivial self-cover is $7$. In \cite[Example 2.5.]{bele03-1}, there is a $7$-dimensional example of nilpotency class $6$ and a natural question is for which nilpotency classes this minimal dimension can be obtained. We have the following example in nilpotency class $5$:
\begin{Ex}
\label{nilp5}
Let $\lien$ be the $5$-step nilpotent rational Lie algebra with basis $X_1, X_2, \ldots, X_7$ and Lie bracket given by 
\begin{align*}
\left[X_1,X_2 \right] &= X_3  &\left[X_1,X_5 \right] &= X_7 &\left[X_2,X_4\right] &= X_7\\
\left[X_1,X_3 \right] &= X_4  &\left[X_1,X_6 \right] &= X_7 &\left[X_2,X_5\right] &= X_6\\
\left[X_1,X_4 \right] &= X_7  &\left[X_2,X_3\right] &= X_5  &\left[X_3,X_5\right] &= X_7.
\end{align*}
A computation shows that $\lien$ is characteristically nilpotent, so every full subgroup of the corresponding radicable hull $N^\Q$ is co-Hopfian. 
\end{Ex}
\noindent Unfortunately, there is only a classification of complex Lie algebras of dimension $7$ (see e.g.\ \cite{magn95-1}), but no classification over $\Q$. Since all complex nilpotent Lie algebras of dimension $7$ of nilpotency class $3$ or $4$ are positively graded, it seems likely that Example \ref{nilp5} has minimal nilpotency class. For nilpotency class $3$ and $4$, there are $8$-dimensional examples of Lie algebras which are characteristically nilpotent.

Another interesting class of examples are Lie algebras without expanding automorphisms, but which do have a non-negative and non-trivial grading. The minimal dimension of such an example is again $7$, because of Theorem \ref{dimension}. 
\begin{Ex}
\label{notcohopf}
Consider the Lie algebra $\lien$ with basis $X_1, \ldots, X_7$ and Lie bracket:
\begin{align*}
\left[X_1,X_2 \right] &= X_3  &\left[X_2,X_3 \right] &= X_5 &\left[X_2,X_6\right] &= X_7\\
\left[X_1,X_3 \right] &= X_4  &\left[X_2,X_4 \right] &= X_6 &\left[X_3,X_5\right] &= -X_7\\
\left[X_1,X_5 \right] &= X_6  &\left[X_2,X_5 \right] &= X_7.
\end{align*}
As an exercise, one can check that $\lien$ has no expanding automorphisms. But the map $\varphi: \lien \to \lien$ defined by
\begin{align*}
\varphi(X_1) &= X_1  & \varphi(X_5) &= 4X_5 \\
\varphi(X_2) &= 2X_2  & \varphi(X_6) &= 4X_6 \\
\varphi(X_3) &= 2X_3  & \varphi(X_7) &= 8X_7 \\
\varphi(X_4) &= 2X_4
\end{align*}
is an automorphism of $\lien$ with characteristic polynomial in $\Z[X]$ and $\det(\varphi) = 2^{10} >0$.
\end{Ex}
Again, if all $7$-dimensional rational Lie algebras of nilpotency class $3$ and $4$ are graded, this is an example of minimal nilpotency class. 

\smallskip

We present a general way of constructing nilmanifolds with various properties starting from an arbitrary $\Q$-group. Let $\lien$ be a rational nilpotent Lie algebra and consider the projection map $\pi: \Aut(\lien) \to \Aut\left(\faktor{\lien}{[\lien,\lien]}\right)$, with $ \Aut\left(\faktor{\lien}{[\lien,\lien]}\right) \approx \GL(n,\Q)$ for some $n \in \N$. The kernel of $\pi$ consists of unipotent automorphisms, see \cite{gs84-1}, and the image $H$ is a $\Q$-group of $\GL(n,\Q)$.  We don't go into details about $\Q$-groups, but the reader can find more information in \cite{bore91-1}. 

Let $\varphi \in \Aut(\lien)$ be an automorphism and denote by $\lambda_1, \ldots, \lambda_n$ the eigenvalues of $\pi(\varphi)$. Every eigenvalue $\mu$
of $\varphi$ can then be written as an $i$-fold product $\lambda_{j_1} \ldots \lambda_{j_i}$ for some $i \in \{1, \ldots, c\}$, where $c$ is the nilpotency class of $\lien$. Thus $\varphi$ is an expanding automorphism if and only if $\pi(\varphi)$ is an expanding automorphism. So a nilmanifold admits an expanding map if and only if the $\Q$-group $H$ of the corresponding Lie algebra $\lien$ has an expanding automorphism. Similarly, the existence of a non-trivial self-cover is equivalent to the existence of an element in $H$ with characteristic polynomial in $\Z[X]$ and determinant $> 1$ in absolute value. 

Starting from a $\Q$-group $H \subseteq \GL(n,\Q)$ with $n \geq 2$, \cite{bg86-1} shows that there always exists a rational Lie algebra $\lien$, generated by $n$ elements, such that $H$ is the image of the projection map $\pi$. So to construct nilmanifolds without non-trivial self-covers, it suffices to construct a $\Q$-group $H$ not containing elements with characteristic polynomial in $\Z[X]$ and determinant $> 1$ in absolute value:

\begin{Ex}
Consider the $\Q$-closed group \begin{align*}
H = \left\{  \begin{pmatrix}q & 0 & 0\\0 & q & 0\\0 & 0 & q^{-1}  \end{pmatrix} \mid q \in \Q \right\}
\end{align*}
and $\lien$ a rational Lie algebra with $H$ as image of the projection map $\pi$. Only the identity element of $H$ has characteristic polynomial in $\Z[X]$, so every nilmanifold with Lie algebra $\lien$ is co-Hopfian. Note that $\lien$ does have automorphisms with determinant $\neq 1$ in absolute value, so this is a new type of example for co-Hopfian $\mathcal{F}$-groups.
\end{Ex}

\begin{Ex}
Consider the $\Q$-closed group \begin{align*}
H = \left\{  \begin{pmatrix}1 & 0\\0 & q\end{pmatrix}\mid q \in \Q  \right\}
\end{align*}
and $\lien$ a rational Lie algebra with $H$ as image of the projection map $\pi$. Every element of $H$ has eigenvalue $1$, so $\lien$ doesn't admit an expanding automorphism. There does exists an automorphism with only eigenvalues in $\Z$ and determinant $>1$ in absolute value, for example any automorphism $\varphi$ with $\pi(\varphi) = \begin{pmatrix} 1 & 0\\0 & 2\end{pmatrix}$. So every nilmanifold with Lie algebra $\lien$ has non-trivial self-covers but no expanding maps. In fact, the Lie algebra of Example \ref{notcohopf} is an explicit example of a Lie algebra with image of $\pi$ equal to $H$.
\end{Ex}

\smallskip

Every known example of an infra-nilmanifold with an Anosov diffeomorphism also admits an expanding map. From the algebraic characterization, there is no reason why having an Anosov diffeomorphism would imply the existence of an expanding map. But since examples of nilmanifolds with Anosov diffeomorphisms are hard to construct (see for example \cite{dere13-1}), it is not easy to give an example of an infra-nilmanifold with Anosov diffeomorphism and not admitting an expanding map although we believe that such examples exist. A problem with the construction above starting from a $\Q$-group $H$ is that the existence of an Anosov diffeomorphism doesn't depend solely on the group $H$. An example would be easier to construct if the conjectures formulated in Section \ref{algexp} above are true.

\bibliography{G:/algebra/ref}

\begin{thebibliography}{10}

\bibitem{bele03-1}
Belegradek, I.
\newblock {\em On co-{H}opfian nilpotent groups}.
\newblock Bull. London Math. Soc., 2003, 35 6, 805--811.

\bibitem{bore91-1}
Borel, A.
\newblock {\em Linear algebraic groups}, volume 126 of {\em Graduate Texts in
  Mathematics}.
\newblock Springer-Verlag, second edition, 1991.

\bibitem{bg86-1}
Bryant, R.~M. and Groves, J. R.~J.
\newblock {\em Algebraic groups of automorphisms of nilpotent groups and {L}ie
  algebras}.
\newblock J. London Math. Soc. (2), 1986, 33 3, 453--466.

\bibitem{degr07-1}
de~Graaf, W.~A.
\newblock {\em Classification of 6-dimensional nilpotent {L}ie algebras over
  fields of characteristic not 2}.
\newblock J. Algebra, 2007, 309 2, 640--653.

\bibitem{deki96-1}
Dekimpe, K.
\newblock {\em {A}lmost-{B}ieberbach {G}roups: {A}ffine and Polynomial
  Structures}, volume 1639 of {\em Lect. Notes in Math.}
\newblock Springer--Verlag, 1996.

\bibitem{deki99-1}
Dekimpe, K.
\newblock {\em Hyperbolic automorphisms and {A}nosov diffeomorphisms on
  nilmanifolds}.
\newblock Trans. Amer. Math. Soc., 2001, 353 7, pp.\ 2859--2877.

\bibitem{dd13-1}
Dekimpe, K. and Der\'{e}, J.
\newblock {\em Existence of Anosov diffeomorphisms on infra-nilmanifolds
  modeled on free nilpotent Lie groups}.
\newblock arXiv:1304.6529.

\bibitem{dip03-1}
Dekimpe, K., Igodt, P., and Pouseele, H.
\newblock {\em Expanding Automorphisms and Affine Structures on Nilpotent Lie
  Algebras with Few Generators.}
\newblock Comm. Algebra, 2003, 31 (12), pp. 5847--5874.

\bibitem{dl03-2}
Dekimpe, K. and Lee, K.~B.
\newblock {\em Expanding maps, {A}nosov diffeomorphisms and affine structures
  on infra-nilmanifolds}.
\newblock Topology Appl., 2003, 130 3, 259--269.

\bibitem{dl03-1}
Dekimpe, K. and Lee, K.~B.
\newblock {\em Expanding maps on infra-nilmanifolds of homogeneous type}.
\newblock Trans. Amer. Math. Soc., 2003, 355 (3), pp. 1067--1077.

\bibitem{dv09-1}
Dekimpe, K. and Verheyen, K.
\newblock {\em Anosov diffeomorphisms on nilmanifolds modeled on a free
  nilpotent Lie group}.
\newblock Dynamical Systems -- an international journal, 2009, 24 1, pp.
  117--121.

\bibitem{dv11-1}
Dekimpe, K. and Verheyen, K.
\newblock {\em Constructing infra-nilmanifolds admitting an {A}nosov
  diffeomorphism}.
\newblock Adv. Math., 2011, 228 6, 3300--3319.

\bibitem{dere13-1}
Der\'{e}, J.
\newblock {\em New examples of Anosov diffeomorphisms on nilmanifolds}.
\newblock arXiv:1312.2872.

\bibitem{grom81-1}
Gromov, M.
\newblock {\em Groups of polynomial growth and expanding maps}.
\newblock Institut des Hautes \'Etudes Scientifiques, 1981,  53, pp. 53--73.

\bibitem{gs84-1}
Grunewald, F. and Segal, D.
\newblock Reflections on the classification of torsion-free nilpotent groups.
\newblock In {\em Group theory}, pages 121--158. Academic Press, London, 1984.

\bibitem{magn95-1}
Magnin, L.
\newblock {\em Adjoint and Trivial Cohomology Tables for Indecomposable
  Nilpotent Lie Algebras of Dimension $\le 7$ over $\C$}.
\newblock Electronic Book, 1995.

\bibitem{mann74-1}
Manning, A.
\newblock {\em There are no new {A}nosov diffeomorphisms on tori}.
\newblock Amer. J. Math., 1974, 96 3, pp. 422--429.

\bibitem{newh70-1}
Newhouse, S.~E.
\newblock {\em On codimension one {A}nosov diffeomorphisms}.
\newblock Amer. J. Math., 1970, 92 761--770.

\bibitem{sega83-1}
Segal, D.
\newblock {\em Polycyclic {G}roups}.
\newblock Cambridge University Press, 1983.

\bibitem{smal67-1}
Smale, S.
\newblock {\em Differentiable dynamical systems}.
\newblock Bull. Amer. Math. Soc., 1967, 73, pp. 747--817.

\end{thebibliography}
\bibliographystyle{G:/algebra/ref}

\end{document}